\documentclass[11pt]{amsart}
\usepackage{geometry}                
\usepackage{graphicx}
\usepackage{amssymb}
\usepackage{epstopdf}
\usepackage{amsaddr}
\makeatletter
\def\blfootnote{\xdef\@thefnmark{}\@footnotetext}
\makeatother

\newtheorem{Theorem}{Theorem}[section]

\newtheorem{Claim}{Claim}[section]

\newtheorem{Fact}{Fact}[section]

\newtheorem{Definition}{Definition}[section]

\newtheorem{Question}{Question}[section]

\title{Borel Chain Conditions of Borel posets}
\author{Ming Xiao}
\begin{document}
\maketitle

\begin{abstract}
We study the coarse classification of partial orderings using chain conditions in the context of descriptive combinatorics. We show that (unlike the Borel counterpart of many other combinatorial notions), we have a strict hierarchy of different chain conditions, similar to the classical case. 
\end{abstract}

\section{introduction}


Let $X$ be a Polish space. A partial order $<$ over $X$ is said to be a Borel partial order if it is a Borel subset of $X^2$. This class of partial orders has been found to play a central role in the theory of forcing, and particularly in the theory of cardinal characteristics of continuum. The first systematical study of these posets is by Harrington, Marker and Shelah in \cite{HMS}, in which they observed a typical dichotomy:
\begin{Theorem}\cite{HMS}\label{HMS}
If $(X,<)$ is a Borel partial order, then either:
\begin{enumerate}
    \item it is a union of countably many Borel chains, or
    \item it includes a perfect pair-wise incomparable subset. 
\end{enumerate}
\end{Theorem}

If we let $(X,E)$ be the incomparability graph of $\leq$ (i.e. $E=X^2\setminus (<\cup\geq)$), the above theorem can be restated as: either $(X,E)$ has countable Borel chromatic number or it includes a perfect complete graph. This statement is in the same spirit of the $G_0$-dichotomy of Kechris, Solecki and Todorcevic in theory of Borel chromatic number:

\begin{Theorem}\cite{KST}\label{KST}
There is a Borel graph $G_0$ on $2^\omega$ such that for every analytic graph $G$ on a Polish space $X$, exactly one of the following holds:
\begin{enumerate}
    \item $X$ has countable Borel chromatic number, or
    \item there is a continuous map from $2^\omega$ into $X$ preserving edges(i.e. its square sends $G_0$ into $G$). 
\end{enumerate}
\end{Theorem}

In fact, Theorem \ref{HMS} can be proved as a corollary of $G_0$-dichotomy (see, e.g., \cite{Mil12}). \\


The incomparability graph is not the only combinatorial notion that draws our attention. The main focus of this paper is the incompatibility graph, which, on the first sight, seems similar to the incomparability graph. However the phenomenon we are going to observe only belong to the incompatibility graph.

Our subject is based on the following notions:

\begin{Definition}
Let $P$ be a poset and $A\subset P$.
\begin{enumerate}
    \item Let $n>1$ be an integer. $A$ is $n$-linked if for every subset $A'\subset A$ of size $n$, there is $z\in P$ so that $z\leq x$ for all $x\in A'$.
    \item $A$ is linked if it is $2$-linked.
    \item $A$ is centred if it is $n$-linked for all $n>1$. 
    \item $x,y\in P$ are compatible if the set $\{x,y\}$ is linked.
    \item $x,y\in P$ are incompatible if they are not compatible.
    \item $A$ is an antichain if it is pairwise incompatible.
\end{enumerate} 

\end{Definition} 

the chain condition method is a way of classifying partial orders by looking at the certain combinatorial properties of compatibilities and incompatibilities. The importance of the chain conditions was first noticed in the characterization of topologies on linearly ordered sets, and was quickly applied in the measure theory and in the theory of forcing. 


Among the many chain conditions have been studied, here is a list of most:

\begin{Definition}
Let $P$ be a poset. 
\begin{enumerate}
    \item $P$ satisfies the $\sigma$-finite chain condition if there is a countable partition $P=\bigcup_n P_n$ so that each $P_n$ includes no infinite antichain.
    \item $P$ satisfies the $\sigma$-bounded chain condition if there is a countable partition $P=\bigcup_n P_n$ so that each $P_n$ includes no antichains of size $\geq n$.
    \item $P$ is $\sigma$-$n$-linked if there is a countable partition $P=\bigcup_k P_k$ so that each $P_k$ is $n$-linked.
    \item $P$ is $\sigma$-centred if there is a countable partition $P=\bigcup_n P_n$ so that each $P_n$ is centred. 
\end{enumerate}
\end{Definition}

While these conditions are obviously listed from weaker to stronger, the fact that their strength is strictly increasing is non-trivial--especially for $\sigma$-finite chain condition and $\sigma$-bounded chain condition which were first studied and conjectured to be different in \cite{horntarski}, and whose strength was just differentiated during the last decade in \cite{thummel}(also see \cite{to2014} for a Borel solution). 

In this work, we study these chain conditions on Borel partial orders defined on Polish spaces and restrict ourselves to only Borel witnesses. Namely, we study following list of properties:

\begin{Definition}
Let $(P,\leq)$ be a Borel poset(i.e. $P$ is a Polish space or a standard Borel space, the partial order $\leq$ is a Borel subset of $P^2$).
\begin{enumerate}
    \item $P$ satisfies Borel $\sigma$-finite chain condition if there is a countable partition $P=\bigcup_n P_n$ so that each $P_n$ is Borel and includes no infinite antichain.
    \item $P$ satisfies Borel $\sigma$-bounded chain condition if there is a countable partition $P=\bigcup_n P_n$ so that each $P_n$ is Borel and includes no antichains of size $\geq n$.
    \item $P$ is Borel $\sigma$-$n$-linked if there is a countable partition $P=\bigcup_n P_n$ so that each $P_n$ is Borel and $n$-linked.
    \item $P$ is Borel $\sigma$-centred if there is a countable partition $P=\bigcup_n P_n$ so that each $P_n$ is Borel and centred. 
\end{enumerate}
\end{Definition}

Our main theorem states that this hierarchy is indeed a non-trivial one:

\begin{Theorem}\label{main}
All properties listed above are distinct. 
\end{Theorem}

As mentioned above, such a non-trivial hierarchical strucutre would not occur in the theory of  incomparabilities of a Borel posets. On the other hand, these Borel chain conditions is also significantly different from the classical ones. As we will see, all examples that differentiates this hierarchy can be taken to be $\sigma$-centred.

We follow standard notations in descriptive set theory. See, e.g., \cite{Ke}. 



\section{preparation}

For a set $X$, a (symmetric) hypergraph over $X$ is a pair $(X,H)$ where $H$ (called the set of edges) is a subset of $[X]^{<\omega}\setminus X$. If all edges are of a same finite size $d$, we say $X$ is a $d$-dimensional hypergraph. We will write the pair $(X,H)$ as $X$ when there is no confusion which hypergraph structure we are talking about. For a hypergraph $(X,H)$, a subset $A\subset X$ is called an anti-clique if there is no subset $A'\subset A$ satisfies $A'\in H$. $(X,H)$ is called a Borel hypergraph if $X$ is a Polish space and $H$ a Borel subset of $[X]^{<\omega}$ equipped with the usual product topology. The Borel chromatic number $\chi_B(X,H)$ is the smallest cardinality of a Polish space to which there is a Borel map being non-constant on every edge.

We are going to heavily use the concepts related to trees. In this work, a (order theoretical) tree $T$ is always a subset of $\omega^{<\omega}$. Given a tree $T$, we denote by $[T]$ the set of all its infinite branches. 

\begin{Definition}
Let $X$ be a hypergraph. Denote as $P(X)$ the poset of all finite anti-cliques of $X$, ordered by reverse inclusion. 
\end{Definition}

Note that when $(X,H)$ is a Borel hypergraph, $P(X)$ is a Borel poset. 

The hypergraphs we are going to use are defined on the set of branches of several trees. Let $T_n$ be the tree $n^{<\omega}$ for each $n$ and $T_{\infty}=\bigcup_{n<\infty} n^n$. For each $T_n$ and $T_{\infty}$, fix a subset $D_n\subset T_n$, $D_{\infty}\subset T_{\infty}$, respectively, so that each of them is dense and intersects each level with exactly one node.

For each $n$, define $H_n=\{\{d\frown i\frown x:0\leq i<n\}:d\in D_n, x\in [T_n]\}$ to make each $[T_n]$ a $n$-dimensional hypergraph. In the same spirit, let $H^0_{\infty}=\{\{d\frown i\frown x:0\leq i<|d|\}:d\in D_{\infty}, x\in [\bigcup_{n>|d|+1} n^n]\}$ and $H^1_{\infty}=\{\{d\frown i\frown x, d\frown j\frown x\}:0\leq i\neq j<|d|, d\in D_{\infty}, x\in [\bigcup_{n>|d|+1} n^n]\}$. 

These hypergraphs naturally generalize graphs $G_0$ defined in \cite{KST} and are well-studied in descriptive combinatorics (see, for example, \cite{Lecomte09}). One of important properties is that they have uncountable Borel chromatic number:
\begin{Fact}\label{uncblechi}
The hypergraphs $([T_n], H_n)$, $([T_{\infty}], H^0_{\infty})$ and $([T_{\infty}], H^1_{\infty})$ are all of uncountable Borel chromatic number. Moreover, for every countable partition of $([T_{\infty}], H^1_{\infty})$ into Borel subsets, one fragment includes complete subgraphs of arbitrarily large sizes. 
\end{Fact}

Here we go through a standard argument using the property of Baire for the case $[T_n]$, other cases follow from the same method. 

\begin{proof}

Suppose not. Then there is a Borel map $f:[T_n]\to \mathbb{N}$ that is non-constant on every edge. As $[T_n]$ is a Polish space, there must be an integer $k$ and a node $t\in D_n$ so that $f^{-1}(k)$ is comeager in the basic open set $\{t\frown x:x\in [T_n]\}$. Therefore, for each $0\leq i<k$, $f^{-1}(n)$ is comeager in $\{t\frown i\frown x:x\in [T_n]\}$. Let $U_i=\{x:t\frown i\frown x\in f^{-1}(k)\}$. Each $U_i$ is comeager in $[T_n]$ thus they have a non-empty intersection. Take an $x$ from this intersection, $\{t\frown i\frown x:0\leq i<n\}$ form an edge on which $f$ is constant, a contradiction.

\end{proof}



\section{proof of the theorem}

The purpose of this section is to show the following facts:
\begin{enumerate}
    \item $P([T_{\infty}],H^1_{\infty})$ is Borel $\sigma$-finite-c.c., but not Borel $\sigma$-bounded c.c..
    \item $P([T_2], H_2)$ is Borel $\sigma$-bounded c.c., but not Borel $\sigma$-linked.
    \item For every $n>2$, $P(([T_n],H_n))$ is Borel $\sigma$-$(n-1)$-linked, but not Borel $\sigma$-$n$-linked.
    \item $P([T_{\infty}],H^0_{\infty})$ is Borel $\sigma$-$n$-linked for every $n>2$, but not Borel $\sigma$-centred.
\end{enumerate}
which clearly together imply theorem \ref{main}. \\

First we show the ``not" part:

\begin{proof}

For each hypergraph $[T]$ mentioned, we can naturally identify $[T]$ with the subset of $P([T])$ consisting with all singletons. For $T$ being $2$-dimension ($[T_2]$ and $([T_{\infty}], H^1_{\infty})$), every complete subgraph is an antichain. In $[T_n]$, a subset is centred if and only if it is an anti-clique. In $([T_{\infty}],H^0_{\infty})$, a subset is $n$-linked if and only if it does not include any edge of size $\leq n$. Then the ``not" part follows from fact \ref{uncblechi}.

\end{proof}

Firstly, we describe a construction of Borel partitions for $P(([T_n],H_n))$ witnessing Borel $\sigma$-$(n-1)$-linkedness, for $n>2$. For the rest three posets, the partition would be the same but we need to reasoning slight differently to see why they work. 

\begin{proof}

For each hypergraph $X$, let $P_k(X)=\{p\in P(X):|p|=k+1\}$. Clearly each $P_k(X)$ is Borel if $X$ is, and $P(X)=\bigcup P_n(X)\cup \{\emptyset\}$.

\begin{Claim}\label{Claim}

For each hypergraph $[T_n]$, for each $\{x_0,...,x_k\}\in P_k([T_n])$, there are $k+1$ distinct nodes $\{t_0,...,t_k\}\subset T_n$ of the same height such that $t_i\sqsubset x_i$ and for every tuple $\{y_0,...,y_k\}\subset [T_n]$ satisfying $t_i\sqsubset y_i$ for every $0\leq i\leq k$, $\{y_0,...,y_k\}$ is an anticlique. 

\end{Claim}

\begin{proof}
Fix $[T_n]$. For every two branches $x,y$ in a tree, denote by $\Delta(x,y)$ the longest initial segment of $x$ and $y$. Fix $\{x_0,...,x_k\}\in P_k([T_n])$. For it to be an anti-clique, it must fall into one of three cases: 

\begin{enumerate}
	\item there are $i\neq j$ so that $\Delta(x_i,x_j)\notin D_n$, or
	\item there are $i\neq j\neq k$ so that $\Delta(x_i, x_j)\neq \Delta(x_j, x_k)$, or
	\item there is a $d\in D_n$ so that for every $i\neq j$ we have $\Delta(x_i,x_j)=d$ but there are $i\neq j$ and $l_{ij}>|d|$ so that $x_i(l_{ij})\neq x_j(l_{ij})$. 
\end{enumerate}  

In first two cases, let $l=sup\{|\Delta(x_i,x_j)|+1\}_{0\leq i\neq j\leq k}$. If the third case happens, pick such $i\neq j$ and let $l=l_{ij}+1$. Let $t_i=x_i|l$ (the initial segment of $x_i$ of length $l$). Then each tuple in the open set $\{\{y_i\}_{0\leq i\leq k}:t_i\sqsubset y_i\}$ realizes the same case as $\{x_0,...,x_k\}$ below level $l$, and thus is an anti-clique. 

\end{proof}

Now for each $p\in P_k([T_n])$ we pick $S_p=\{t_0,...,t_k\}$ and let $U_p$ be the open neighbourhood of $P_k([T_n])$ defined by $U_p=\{\{y_i\}_{0\leq i<k}:t_i\sqsubset y_i\}$. Clearly, $U_p$ is Borel (in fact it is open). We show that it is $n-1$-linked. Let $A=\{p_0,...,p_{n-2}\}\subset U_p$ be a subset of size $n-1$. We show that it is centred (i.e. their union is still an anti-clique): 

if not, take $(x_0,...,x_{n-1})$ being an edge in $\bigcup_{i<l} p_i$. By above claim, there are is $i\leq k$, and there are $x_{j_0}$ and $x_{j_1}$ both extending the $t_i$, thus $|\Delta(x_{j_0},x_{j_1})|\geq t_i$. On the other hand, by pigeon hole principal (and the fact that $n>n-1$), there has to be $m<n_1$, $x_{l_0}\neq x_{l_1}$ both in $p_m$. By our definition of $U_p$, $|\Delta(x_{l_0},x_{l_1})|<|t_i|$. Thus we have $\Delta(x_{j_0},x_{j_1})\neq \Delta(x_{l_0},x_{l_1})$. 

However, by the definition of $H_n$, $\Delta (x_{i_0},x_{i_1})=\Delta(x_{j_0},x_{j_1})$ for every pairs $i_0\neq i_1$ and $j_0\neq j_1$. This contradiction shows that there cannot be any edge in  $\bigcup_{0\leq i<n-1} p_i$. 

Lastly, notice that while there are uncountably many $p$, there can only be countably many $S_p$ since they are finite subsets of the countable set $T_n$. Also, it is clear that $p\in U_p$, thus $P([T_n])=\bigcup_{p\in P} U_p$ is actually a countable partition of $P[T_n]$ into countably many $n-1$-linked Borel subsets, as wanted. 

\end{proof}

Now we turn to the case $P([T_2],H_2)$. The Claim from above still works for $n=2$ so we can still construct $U_p$. For this case, we want to show that each $U_p$ only includes antichains of bounded size. Fix $p\in P([T_2],H_2)$, pick $S_p$ as in the above proof and let $A\subset U_p$ be an antichain. Order $S_p=\{t_0,...,t_k\}$.  For each pair $q_0\neq q_1\in A$, as they are incompatible, there has two be an $H_2$ edge connecting $x\in q_0, y\in q_1$. By our claim, there has to be $0\leq i\leq k$ so that $x$ and $y$ both end-extends $t_i$. We color this (unordered) pair $q_0,q_1$ with the least such $i$. $A$ is an antichain, so $[A]^2$ is fully colored. By the Ramsey theorem, when $|A|$ is large enough (more precisely, when it is no less than the $|p|$-color Ramsey number $R(3,3,...,3)$), there are $t\in S_p$, $p_0\neq p_1\neq p_2\in A$ and $x_i\in p_i$ so that $t\sqsubset x_i$ for $i=0,1,2$ and $x_0,x_1,x_2\in [T_2]$ form a triangle($K_3$). However, this is impossible: it is a well-known fact that $([T_2],H_2)$ (which is just $G_0$) is loop-free. 

Thus $U_p$ does not include any antichain of size larger than the $|p|$-color Ramsey number $R(3,3,..3)$. This number clearly only depends on the size of $p$ and is independent of our choice of $S_p$. Again, there are only countably many different possible $S_p$ so $P([T_2],H_2)=\bigcup_{p\in P} U_p$ is $\sigma$-bounded c.c.

For $P([T_{\infty}],H^0_{\infty})$, we need to construct a partition witnessing Borel $\sigma$-$n$-linkedness for each $n$. For this purpose, we turn back to the Claim \ref{Claim}. In addition to the requirements in the Claim, we also require the $|t_i|>n$. This can be achieved simply by pick $l=n+1$ if the original $l\leq n$(otherwise we can just remain it unchanged). Once this is done, the same proof of Borel $\sigma$-$n$-linkedness works for $P([T_{\infty}],H^0_{\infty})$.

Lastly, for $P([T_{\infty}],H^1_{\infty})$, we show that for every $p$ and for any $S_p$ as in the Claim, $U_p$ does not include infinite anti-chains. The proof goes exactly the same as case $P([T_2],H_2)$, only slightly differs at the use of Ramsey theorem: Instead of $K_3$, this time we use Ramsey theorem to pick an infinite complete subgraph $G$ from $([T_{\infty}],H^1_{\infty})$. We now show that there cannot be any infinite complete subgraph. Pick $x\in G$. By the definition of $H^1_{\infty}$, for each $d\in D_{\infty}$ there are only finitely many $x'\in [T_{\infty}]$ satisfying $x'\in G$ and $\Delta(x,x')=d$. Therefore there has to be $y\neq z$ so that $\Delta(x,y)=d_0\sqsubset \Delta(x,z)=d_1$ for $d_0\neq d_1\in D_{\infty}$. In this case, we can see that $\Delta(y,z)=d_0$ as well. Since $y,z\in G$, there has to be integers $i\neq j$ and real $r$ so that $y=d_0\frown i\frown r$ and $z=d_0\frown j\frown r$. However, this implies that $d_0\frown j\sqsubset d_1\sqsubset x$. By looking at the definition of $H^1_{\infty}$ again, we notice that an edge from $x$ to $y$ makes $x=d_0\frown j\frown r=z$, contradicting our choice of $x,y,z$ to be distinct.

\section{Comparison with classical cases}

It worth notice that every poset we mentioned above are all $\sigma$-centred if we do not require the fragmentation to be Borel.

\begin{Theorem}
If $H$ is a hypergraph with at most continuumly many connected components and each connected component is countable, then $P(H)$ is $\sigma$-centred.
\end{Theorem}

\begin{proof}
Let $H=\bigcup_{\lambda<\mathfrak{c}} H_{\lambda}$ where each $H_{\lambda}$ is a connected component of $H$.  Equip it with the discrete topology and consider the topological space $X=\Pi_{\lambda<\mathfrak{c}} P(H_{\lambda})$ equipped with the usual product topology. Every $P(H_{\lambda})$ is countable thus in particular separable. By the Hewitt-Marczewski-Pondiczery theorem $X$ is also separable. Take $D\subset X$ be a countable dense subset. For each $d_n\in D$, let $P_n=\{p:p\in P(H)$ and $p\cap H_{\lambda}=d_n(\lambda)\}$. Every $P_n$ is centred since $\bigcup_{\lambda<\mathfrak{c}} d_n(\lambda)$ is an anti-clique in $H$. Also for every $p\in P(H)$, the subset $\{x:x(\lambda)=p\cap H_{\lambda}$ or $p\cap H_{\lambda}=\emptyset$ for all $\lambda\}$ is open, so there is a $d_n$ in it, and equivalently, $p\in P_n$. 
\end{proof}

Our hypergraphs $[T_n]$, $([T_{\infty}], H^0_{\infty})$ and $([T_{\infty}], H^1_{\infty})$ all satisfy the requirement of the above theorem since any two vertices in an edge are eventually equal, thus all posets we dealt with are $\sigma$-centred. 

For another interesting example that fails Borel $\sigma$-finite chain condition and the usual $\sigma$-bounded chain condition but satisfies $\sigma$-finite chain condition, see \cite{MX}.

\section{further observations}

When we look at the $\sigma$-bounded chain condition, a naturally aroused question is whether replacing ``bounded" with ``uniformly bounded" would result a new property that lie strictly in between $\sigma$-bounded chain condition and $\sigma$-linkedness or not. More precisely, we consider the following property:

\begin{Definition}
Let $n$ be a positive integer. A poset $P$ is said to satisfy the $\sigma$-$n$-chain condition if there is a countable partition $P=\bigcup P_i$ so that for every $i$, every antichain $A\subset P_i$ has size $<n$. When $P$ is a Borel poset and $P_i$ can be taken to be Borel simultaneously, we say that $P$ satisfies the Borel $\sigma$-$n$-chain condition.
\end{Definition}

However, the following (unpublished, as far as the author knows) theorem of Galvin and Hajnal states that this property is actually just $\sigma$-linkedness:

\begin{Theorem}[Galvin, Hajnal]\label{GH}
For any positive integer $n$, a poset $P$ satisfies $\sigma$-$n$-chain condition if and only if it is $\sigma$-linked.
\end{Theorem}

\begin{proof}

Suppose not. Let $P=\bigcup_{i<\omega} P_i$ be a partition witness $\sigma$-$n$-chain condition for the smallest $n$ possible. Note that if $n=2$, then $P$ is already $\sigma$-linked. For the following we assume that $n>2$. 

Since $P$ does not satisfy $\sigma$-$n-1$-chain condition, there must be $k<\omega$ so that for every partition of $P_k=\bigcup_{j<\omega} P_{k,j}$ has a fragment $P_{k,l}$ including an antichain of size no less than $n-1$ (and thus equals to $n-1$). For each $p\in P_k$, let $L(p)=\{q\in P_k$ so that $p$ and $q$ are incompatible$\}$ and $R_i(p)=\{q\in P_k$: there is a $r\in P_i$ extending both $p$ and $q\}$. Then for every $p$, $L(p)\cup(\bigcup_{i}R_i(p))=P_k$. Moreover, for each $p$ there is an integer $i(p)$ so that $R_{i(p)}(p)$ contains an antichain of size $n-1$. For each $i$, let $Q_i=\{p\in P_k: i(p)=i\}$ . Clearly $\bigcup_i Q_i=P_k$, thus there is an $l$ so that $Q_l$ contains an antichain $p_1,..., p_{n-1}$ of size $n-1$. Then for each $i=1,2,... n-1$ we can find antichain $q_{i1}, ..., q_{i(n-1)}\subset R_l$ of size $n-1.$ For each $i=1,..., n-1$ and $j=1,2,...,n-1$, we fix $r_{ij}$ in $P_l$ extending both $ p_i$ and $ q_{ij}.$ Then $\{r_{ij}: i, j=1,2,..., n-1\}$ is an antichain and it is a subset of $P_l.$ Since $(n-1)^2>n$, we have a contradiction with $n>2$.

\end{proof}

By the same proof, with a careful tracking of complexity of sets, we can show the same for a wide class of Borel posets:
\begin{Theorem}
Let $P$ be a Borel poset with Borel incompatibility. If there is an integer $k$ and a countable partition $P=\bigcup P_n$ into Borel subsets so that for every $n$, every antichain $A\subset P_n$ has size $<k$, then $P$ is Borel $\sigma$-linked.
\end{Theorem}
(In some articles, a Borel poset with Borel incompatibility is also called a ``Souslin forcing")
\begin{proof}
	First look at the description of linkedness of a subset $A$: ``for every $x,y\in A$, $x$ and $y$ are compatible". When ``being compatible" is Borel, this condition is $\boldsymbol{\Pi^1_1}$ over $\boldsymbol{\Sigma^1_1}$, therefore reflection lemma implies that we can relax the Borel partition in the definition of $\sigma$-$n$-chain condition to $\boldsymbol{\Sigma^1_1}$ partition. 
	
	Let us now track the complexity of each set occured in the proof of the Theorem \ref{GH} and make sure that every step is still valid when we turn to $\boldsymbol{\Sigma^1_1}$ partitions: each $P_k$ is Borel. For each $p$, $L(p)$ is Borel, $R_i(p)$ is $\boldsymbol{\Sigma^1_1}$. By above argument, the proof still holds and the relation ``$R_i(p)$ contains and antichain of size $n-1$" is $\boldsymbol{\Sigma^1_1}$ over pairs $(p,i)$. For each $i$, $Q_i$ is the $i$'th fibre of this set, thus is still $\boldsymbol{\Sigma^1_1}$. The rest of the proof thus can be proceeded. 
	
\end{proof}

While Borel posets with non-Borel incompatibility do exist (for example, take two disjoint Polish spaces $X\cap Y=\emptyset$ and $U\subset X\times Y$ be a closed set with non-Borel projection on $Y$, regard $U$ as a partial order on $X\cup Y$ results such a poset), it is not known if the requirement of Borel incompatibility can be omitted. 

\begin{proof}
In the previous proof, every $L(p)$ is Borel, $R_i(p)$ and $Q_i$ are analytic. Also notice that ``every antichain has size $<n$" is $\boldsymbol{\Pi}^1_1$ over $\boldsymbol{\Sigma}^1_1$, so by reflection lemma the above proof works for Borel posets with Borel incompatibility. 
\end{proof}

Also, due to the $G_0$-dichotomy, the following fact is quickly followed:
\begin{Theorem}
Let $P$ be a Borel poset such that the collection $\mathcal{C}$ of centred subsets is Borel and there is a Borel function $f:\mathcal{C}\to P$ so that for every $C\in\mathcal{C}$, $f(C)\leq p$ for every $p\in C$. Then exactly one of following follows:
\begin{enumerate}
    \item $P$ is Borel $\sigma$-linked, or
    \item There is a $P'\subset P([T_2])$ failing Borel $\sigma$-linkedness and for which there is a Borel map $\phi: P'\to P$ that preserves incompatibility. 
\end{enumerate}
\end{Theorem}

\begin{proof}
Let $(P,G)$ be the incompatibility graph over $P$. Then $P$ is Borel $\sigma$-linked if and only if the Borel chromatic number $\chi_B(P)$ is countable. 

When $P$ is not Borel $\sigma$-linked, there is a Borel map $\psi$ that embeds $G_0$($=[T_2]$) into $(P,G)$. Let $P'=\{p\in P([T_2]):\{\psi(x):x\in p\}\in \mathcal{C}\}$. Let $\phi(p)=f(\{\psi(x):x\in p\})$. This $P'$ and $\phi$ are then as required. 

\end{proof}

And similarly we can replace Borel $\sigma$-linkedness and $[T_2]$ with other Borel chain conditions and corresponding posets. We finish with conjecturing the following strengthening of this theorem:

\begin{Question}

Is it true that for every Borel poset $P$, exactly one of the following holds?
\begin{enumerate}
    \item $P$ is Borel $\sigma$-linked, or
    \item There is a Borel map $\phi:P([T_2])\to P$ preserving incompatibility.
\end{enumerate}

\end{Question}

\end{document}